\documentclass[12pt,draft]{amsart}
\usepackage[all]{xy}
\usepackage{amsfonts}
\usepackage{amssymb}

\usepackage{enumerate}

\usepackage{color}

\renewcommand{\le}{\leqslant}
\renewcommand{\ge}{\geqslant}

\newtheorem{teo}{Theorem}[section]
\newtheorem{lem}[teo]{Lemma}
\newtheorem{prop}[teo]{Proposition}

\newtheorem{cor}[teo]{Corollary}

\theoremstyle{definition}
\newtheorem{dfn}[teo]{Definition}
\newtheorem{rk}[teo]{Remark}

\def\<{\langle}
\def\>{\rangle}
\def\ss{\subset}

\def\a{\alpha}
\def\b{\beta}
\def\g{\gamma}
\def\d{\delta}

\def\l{{\lambda}}
\def\r{\rho}

\def\t{\tau}

\def\w{\omega}

\def\G{{\Gamma}}
\def\F{{\Phi}}

\def\C{{\mathbb C}}
\def\R{{\mathbb R}}
\def\Z{{\mathbb Z}}

\def\End{\mathop{\rm End}\nolimits}
\def\Ker{\mathop{\rm Ker}\nolimits}

\def\Aut{\operatorname{Aut}}
\def\Out{\operatorname{Out}}

\def\Id{\operatorname{Id}}

\def\Trace{\operatorname{Trace}}

\def\Inn{\operatorname{Inn}}

\def\Ind{\operatorname{Ind}}
\def\ind{\operatorname{ind}}

\def\1{\mathbf 1}

\def\ola#1{\stackrel{#1}{\longrightarrow}}

\newcommand{\ov}[1]{\overline{#1}}
\newcommand{\til}[1]{\widetilde{#1}}
\newcommand{\wh}[1]{\widehat{#1}}

\newcommand{\Mat}[4]{\left( \begin{array}{cc}
                            #1 & #2 \\
                            #3 & #4
                      \end{array} \right)}

\def\Fix{\operatorname{Fix}}

\def\N{{\mathbb N}}

\begin{document}

\title
{Twisted conjugacy classes in residually finite groups}

\author{Alexander Fel'shtyn}
\address{Instytut Matematyki, Uniwersytet Szczecinski,
ul. Wielkopolska 15, 70-451 Szczecin, Poland and School of
Mathematics, Institute for Advanced Study, Einstein Drive,
\newline
Princeton, NJ 08540 USA} \email{felshtyn@mpim-bonn.mpg.de,
felshtyn@ias.edu, fels@wmf.univ.szczecin.pl}

\author{Evgenij Troitsky}
\thanks{The second author is partially supported by
RFBR Grants  10-01-00257-a and 11-01-90413-Ykp-o$\!\!|$-a
and Grant  2010-220-01-077
(contract 11.G34.31.0005) of the Russian Government}
\address{Dept. of Mech. and Math., Moscow State University,
119991 GSP-1  Moscow, Russia}
\email{troitsky@mech.math.msu.su}
\urladdr{
http://mech.math.msu.su/\~{}troitsky}

\keywords{Reidemeister number, $R_\infty$-group, twisted conjugacy
class, Burnside-{Frobenius} theorem, (non) amenable group,
residually finite group, conjugacy separable group, unitary dual,
regular representation, matrix coefficient, weak containment,
twisted inner representation, rational representation,
isogredience classes, $S_\infty$-group, mapping class group, Gauss
congruences, Coxeter group, lattice, knot group, Golod-Shafarevich
group, positive rank gradient, polyfree group, braid group,
Artin's group, surface group, surface braid group, nilpotent
group, $C^*$-simple group, hyperbolic group, Reidemeister
spectrum, Osin group}
\subjclass[2000]{20C; 
20E45; 
22D10; 
22D25; 
22D30; 
37C25; 
43A20; 
43A30; 
46L; 
47H10; 
54H25; 
55M20
}

\begin{abstract}
We prove for residually finite groups the following long standing
conjecture: the number of twisted conjugacy classes ($g\sim h g
\phi(h^{-1})$) of an automorphism $\phi$ of a finitely generated
group is equal (if it is finite) to the number of finite
dimensional irreducible unitary representations being invariant
for the dual of $\phi$.

Also, we prove that any finitely generated residually finite
non-amenable group has the $R_\infty$ property (any automorphism
has infinitely many twisted conjugacy classes). This gives a lot
of new examples and covers many known classes of such
groups.
\end{abstract}

\maketitle

\section{Introduction}

The following two interrelated problems are among the principal
ones in the theory of twisted conjugacy (Reidemeister) classes in
infinite discrete groups. The first one is the 20-years-old
conjecture on existence of an appropriate twisted
Burnside-Frobenius theory (TBFT), i.e. identification of the
number $R(\phi)$ of Reidemeister classes and the number of fixed
points of the induced homeomorphism $\wh \phi$ on an appropriate
dual object (supposing $R(\phi)<\infty$). The second one is the
problem to outline the class of $R_\infty$ groups (that is
$R(\phi)=\infty$ for any $\phi$).

In this paper important advances in both problems are obtained.
Namely, first, it is proved that TBFT holds for finitely generated
residually finite groups (we take the finite-dimensional part of
the unitary dual as the dual space). Secondly, it is discovered
that finitely generated residually finite non-amenable groups are
$R_\infty$-groups. Several supplementary results of independent
interest are obtained.

The interest in twisted conjugacy relations has its origins, in
particular, in the Nielsen-Reidemeister fixed point theory (see,
e.g. \cite{Jiang,FelshB}), in Selberg theory (see, eg.
\cite{Shokra,Arthur}), and  Algebraic Geometry (see, e.g.
\cite{Groth}). In representation theory twisted conjugacy probably
occurs first in Gantmacher's paper \cite{Gantmacher} (see, e.g
\cite{Springer,Onishik-Vinberg}).

The problem of determining which classes of groups have the
$R_\infty$ property is an area of active research initiated in
\cite{FelHill}.
 In some
situations the $R_\infty$ property (for fundamental groups) has
direct topological consequences. For example, using this property
in \cite{GoWon09Crelle} for any $n\ge 4$ a compact nilmanifold
$M$, $\dim M=n$, is constructed, such that any homeomorphism
$f:M\to M$ is homotopic (for $n\ge 5$ isotopic) to a fixed point
free map. In \cite{DePe2009} it is shown that an infra-nilmanifold
that admits an Anosov diffeomorphism cannot have the $R_\infty$
property since the Reidemeister number of an Anosov diffeomorphism
on an infra-nilmanifold is always finite. Also, as we will see
below it can be used to detect the trivial knot.

The Reidemeister class of $g\in G$ we will denote by $\{g\}_\phi$:
$ \{g\}_\phi:=\{ x g \phi(x^{-1})\in G \: |\: x\in G\}. $ A
$\phi$-\emph{class function} is any function on $G$, which is
constant on Reidemeister classes, or, equivalently, which is
invariant under the twisted action $g \mapsto x g \phi(x^{-1})$ of
$G$ on itself.

\begin{dfn}
Denote by $\wh G$ the set of equivalence classes of unitary
irreducible representations of $G$ and by $\wh G_f$ its part
corresponding to finite-dimensional representations. The class of
$\r$ in $\wh G$ we will denote by $[\r]$. An automorphism $\phi$
of $G$ induces a bijection $\wh\phi:\wh G \to \wh G$ by the
formula $[\wh \phi (\r)]:=[\r\circ\phi]$.
\end{dfn}

The following (first) main theorem is proved as
Theorem \ref{teo:TBFTforRF} below.

\noindent \textbf{Theorem A.} \emph{ Let $\phi:G\to G$ be an
automorphism of a residually finite finitely generated group $G$
with $R(\phi)<\infty$. Then $R(\phi)$ is equal to the number of
$\wh\phi$-fixed points on $\wh G_f$.}

This theorem covers the known case of
almost polycyclic groups \cite{polyc} and
contains f.g. metabelian groups and restricted
wreath products. Some known examples of $\phi$
with $R(\phi)<\infty$ are listed in Section \ref{sec:AppList}.

In the case of $|G|<\infty$ and $\phi=\Id$ it becomes the
celebrated Burnside-Frobenius theorem, which says that the number
of conjugacy classes is equal to the number of equivalence classes
of irreducible unitary representations of $G$. In \cite{FelHill}
it was discovered that this statement remains true for any
automorphism $\phi$ of any finite group $G$. Indeed, $R(\phi)$ is
equal to the dimension of the space of $\phi$-class functions on
this group. Hence, by Peter-Weyl theorem, $R(\phi)$ is identified
with the sum of dimensions $d_\r$ of twisted invariant elements of
$\End(H_\r)$, where $\r$ runs over $\wh G$, and the space of a
representation $\r$ is denoted by $H_\r$. By the Schur lemma,
$d_\r=1$, if $\r$ is a fixed point of $\wh\phi$, and is zero
otherwise. Hence, $R(\phi)$ coincides with the number of fixed
points of $\wh\phi$.

Theorem A has important dynamical consequences (see
\cite{FelshB,FelTro} for an extended discussion). In particular,
suppose, $R(\phi^n)<\infty$, $n\ge 1$, in Theorem A. Then we have
for all $n$ the following \emph{Gauss congruences for the
Reidemeister numbers}, which are important for the theory of the
Reidemeister zeta function:
\begin{equation}\label{eq:congru}
 \sum_{d\mid n} \mu(d)\cdot R(\phi^{n/d}) \equiv 0 \mod n,
\end{equation}
where $\mu(d)$, $d\in\N$, is the {\em M\"obius function}.
Indeed, apply the M\"obius' inversion formula
to the following evident consequence of Theorem A:
$
 R(\phi^n) = \# \: Fix(\wh\phi^n|_{\wh G_f})
 =\sum_{d\mid n} P_d,
$ where $P_n$ denote the number of periodic points of
$\widehat{\phi}$ on $\wh G_f$ of least period $n$. We obtain
(\ref{eq:congru}) with $P_n$ on the right. But
 $P_n$ is always divisible
by $n$, because $P_n$ is exactly $n$ times the number of
orbits of length $n$.

Formula (\ref{eq:congru}) was previously known  in the   special
case of almost polycyclic groups (see \cite{FelTro,polyc}, where
more detail can be found).
 The history of the Gauss congruences for
 integers can be found in
\cite{Zarelua2008Steklo}.
 Gauss congruences for Lefschetz numbers of
 iterations of a continuous map were proved by
 A. Dold \cite{Dold83Inv}.

The second main theorem (proved as Corollary \ref{cor:mainrinf}
below) is as follows.

\noindent
\textbf{Theorem B.} \emph{
Any non-amenable residually finite finitely
generated group is an $R_\infty$ group.}

\noindent (It is enforced in Prop. \ref{prop:usilennonamen}).

\smallskip
This gives a lot of new examples of groups
with $R_\infty$-property. (A list of known
cases of $R_\infty$ and non-$R_\infty$ groups
can be find in  Section \ref{sec:AppList},
to not overload the Introduction.)

Among the new classes we would like to emphasize
the following important ones.

\begin{enumerate}
\item Irreducible infinite non-affine finetely
generated \emph{Coxeter groups.} Indeed,
they are linear \cite{Vinberg71IAN},
hence residually finite \cite{malcev}.
They contain a subgroup which maps onto a
non-amenable group \cite{MargulVinb2000JLieT},
thus they are non-amenable.

\item Lattices (finitely generated
cocompact discrete subgroups) $G$ in direct products
$G_1\times\dots\times G_m$ of linear groups over fields
$k_1,\dots,k_m$ of zero characteristic, where at least one of
$G_i$, say $G_1$, is a non-compact semi-simple Lie group. Indeed,
$G_1$ is non-amenable by \cite{Furstenberg} (see \cite[p.
432]{BekkaHarpeValetteT}), hence $G_1\times\dots\times G_m$ is
non-amenable, and $G$ is non-amenable (see, e.g. \cite[Corollary
G.3.8]{BekkaHarpeValetteT}). By the same argument as in the
previous item, $G$ is residually finite. This class contains
$S$-arithmetic lattices, in particular groups from
(\ref{it:semisla}), Sect. \ref{sec:AppList}.

\item \emph{Fundamental
groups of compact 3-dimensional manifolds}
and \emph{knot groups} with few exceptions.
They are finitely generated and
residually finite by \cite{Hempel} and Perelman's proof
of Thurston's geometrization conjecture \cite{Thurston}.
More precisely, the \emph{JSJ
decomposition} (see, e.g. \cite{Hempel3mds,Scott1983BLMS})
gives parts,
which are either Seifert manifolds, or hyperbolic.
The fundamental group is a graph product of
fundamental groups of these parts.
Thus, by Theorem B, the group
under consideration is $R_\infty$ group if
it is non-amenable. Such groups are characterized
geometrically (via bounded cohomology \cite{NIvanov})
in \cite{FujiOhshi2002,Fujiwara98}.

In particular, via
existence/non-existence of
$R_\infty$ property for knot groups
we can detect unknot.
Indeed, since the trivial knot's
group is abelian, it is not $R_\infty$, while the
other ones have infinite-dimensional second bounded
cohomology
\cite{FujiOhshi2002,Fujiwara98}.
Thus, they are non-amenable and $R_\infty$ groups
by Theorem B.

\item Several important classes are among
non-linear groups:

\begin{enumerate}[(a)]
\item Periodic Golod-Shafarevich groups
are finitely generated residually finite
non-amenable (having quotients with property (T))
\cite{Ershov2011}
(see also  \cite{ErshovSurv}). A linear non-ame\-n\-able group
contains $F_2$ (Tits alternative,
\cite[pp. 145--146]{WehrfritzLinearGroups}).
Thus, periodic Golod-Shafarevich groups
are non-linear.

\item Similarly,
finitely generated infinite torsion groups with positive rank
gradient from \cite{OsinGradient} or positive power $p$-deficiency
from \cite{SchPuc} (see also \cite[Sect. 9]{ErshovSurv}) are
$R_\infty$ groups.

\item $\Aut(F_n)$, $n\ge 3$, is finitely generated
\cite{Nielsen1924}, residually finite (Baumslag, see \cite[Sect.
2]{BassLubotzky83IsraelJM}), non-amenable  and non-linear
\cite{FormanekProcessi1992}. The same is true for ${\rm
Out}\,(F_n)$, $n\ge 4$ (see e.g. \cite{Vogtmann2002}).
\end{enumerate}
\end{enumerate}

Also, the theorem covers iterated semidirect products
$F_{d_l}\rtimes \dots\rtimes F_{d_1}$ of free groups and its
finite extensions (poly-fg-free-by-finite groups), because they
are residually finite (see, e.g. \cite[Theorem 7, p.
29]{mille-71}) and evidently non-amenable (having a non-amenable
subgroup) and finitely generated. This class contains Artin's pure
braid groups, the fundamental group of the complement of any
affine fiber-type hyperplane arrangement and some others, see e.g.
\cite{CohenSuciu,Meier}.

On the other hand, the new result covers
(up to the conjecture about the residually finiteness)
the most extended known class, non-elementary
Gromov hyperbolic groups (\ref{it:gromov}) in
Sect. \ref{sec:AppList}.

Some other new classes
are discussed in Section \ref{sec:AppList}.

In Section \ref{sec:fixpoiReid} we establish important
inequalities, relating the number of fixed points of $\phi$ and
$R(\phi)$. We prove Theorem A in Section \ref{sec:TBFT} and
Theorem B in Section \ref{sec:amentwistinnrinfty}. In Section
\ref{sec:rationalpoints} we introduce (ir)rational representations
and describe their connection with twisted conjugacy classes. In
Section \ref{sec:isogred} we introduce and investigate a new class
of groups: $S_\infty$ groups.

\smallskip
\textbf{Acknowledgments.} The present paper is a part of our joint
research programm in Max-Planck-Institut f\"ur Mathematik (MPIM)
in Bonn. We would like to thank the MPIM for its kind support and
hospitality while the greater part of this work was completed. The
first author is indebted to the Institute for Advanced Study,
Princeton, and Institut des Hautes \'Etudes Scientifiques
(Bures-sur-Yvette) for for their support and hospitality and the
possibility of the present research during his visit there.

   The authors are grateful to
V. Manuilov, A. Shtern, and A. Vershik
   for important discussions on non-commutative harmonic analysis
   aspects of the present research and to
M. Abert, V. Bardakov, P. Bellingeri, D. Gon\c{c}alves,
   R. Grigorchuk, M. Kapovich, V. Nekrashevich,
   D. Osin, G. Prasad, and M. Sapir
   for helpful
comments and   valuable suggestions on different group-theoretic
questions. The paper was partially inspired by the excellent
survey \cite{Harpe2007simple} by P. de la Harpe and the
interesting paper \cite{Jabara} by E. Jabara.

\section{Preliminary Considerations}\label{sec:prelim}
We start from the following estimation (supposed to be known to
E.~Landau (cf. \cite{Landau1903}) and R.~Brauer (cf. \cite{Brauer63}))
(cf. \cite{NewmanMorris}).
\begin{lem}\label{lem:estimnumbtheor}
Let $x_1\le x_2 \le \dots \le x_n$ be positive integers such that
$\sum_{i=1}^n (x_i)^{-1}=1$. Then $x_n \le n^{2n-1}$.
\end{lem}

\begin{proof}
One has
\begin{equation}\label{eq:estimnumbtheor1}
    x_1\le n,
\end{equation}
since otherwise $\sum_{i=1}^n \frac 1{x_i}<\sum_{i=1}^n \frac 1n = 1$.
We can write
with some positive integers $y_r$ for $1\le r \le n-1$:
$$
\frac{n-r}{x_{r+1}}\ge \sum_{i=r+1}^n \frac 1{x_i} =
1- \sum_{i=1}^r \frac 1{x_i} =\frac{y_r}{x_1x_2 \dots x_r} \ge
\frac{1}{x_1x_2 \dots x_r},
$$
\begin{equation}\label{eq:estimnumbtheor2}
    n x_1x_2 \dots x_r \ge (n-r)x_1x_2 \dots x_r \ge x_{r+1}.
\end{equation}
By induction, (\ref{eq:estimnumbtheor1}) and (\ref{eq:estimnumbtheor2})
imply the statement.
\end{proof}
With some additional efforts the estimation can be improved
(cf. e.g. \cite{NewmanMorris}) but we do not need this.
Let us note that, since $\log_2 n < n$, one has
\begin{equation}\label{eq:sravnnnilog}
    n^{n-1} <n^n =2^{n \log_2 n} < 2^{n^2}.
\end{equation}

Denote the stabilizers related to the twisted action of $G$
on itself by
$$
St^{tw}_{\phi}(g,h):=\{k\in \G \: | \: k g \phi(k^{-1})=h\},
\qquad St^{tw}_{\phi}(g):=St^{tw}_{\phi}(g,g).
$$
In particular,
$$
St^{tw}_{\phi}(e)=\{k\in \G \: | \: k \phi(k^{-1})=e\}
= C_G(\phi)
$$
(fixed elements of $\phi$).
Evidently, $St^{tw}_{\phi}(g,h)=\varnothing$ if $h \not\in \{g\}_\phi$.
Otherwise
$$
St^{tw}_{\phi}(g, s g \phi(s^{-1}))=
\{k\in \G \: | \: k g \phi(k^{-1})=s g \phi(s^{-1})\}
=
\{k\in \G \: | \: s^{-1} k g \phi(k^{-1}s)= g  \},
$$
i.e. $St^{tw}_{\phi}(g, s g \phi(s^{-1}))=s\cdot St^{tw}_{\phi}(g)$
is a coset of this group. Thus
\begin{equation}\label{eq:stabilcoset}
|St^{tw}_{\phi}(g, s g \phi(s^{-1}))|= |St^{tw}_{\phi}(g)|.
\end{equation}

An evident consequence of the above statement is the
following lemma (see \cite[Lemma 4]{Jabara}).

\begin{lem}\label{lem:jabarafinoz}
Let $|G|<\infty$, $r=R(\phi)$, $\phi:G\to G$.
Then $|C_G(\phi)|\le r^{r-1}$.
\end{lem}

\begin{proof}
Let $\{g_i\}_\phi$ be distinct Reidemeister classes, $i=1,\dots,r$,
$g_1=e$.
Then
$$
|G|=\sum_{i=1}^r \# \{g_i\}_\phi=
\sum_{i=1}^r \frac {|G|}{|St^{tw}_{\phi}(g_i)|}.
$$
Dividing by $|G|$ and applying Lemma \ref{lem:estimnumbtheor}
we obtain
$$
|C_G(\phi)|=|St^{tw}_{\phi}(e)|\le \max_i |St^{tw}_{\phi}(g_i)| \le r^{r-1}.
$$
\end{proof}

The following fact is well known.

\begin{lem}\label{lem:normalcharfiniteind}
Let $G$ be finitely generated, and $H'\ss G$
its subgroup of finite index. Then there is a characteristic
subgroup $H\ss G$ of finite index, $H\ss H'$.
\end{lem}

\begin{proof}
Since $G$ is finitely generated, there is only finitely many
subgroups of the same index as $H'$
(see \cite{Hall}, \cite[\S\ 38]{Kurosh}).
Let $H$ be their intersection. Then $H$ is characteristic, in
particular normal, and of finite index.
\end{proof}

Let us denote by $\t_g:G\to G$ the automorphism $\t_g(\til g)=g\til g\,g^{-1}$
for $g\in G$. Its restriction to a normal subgroup we will denote by $\t_g$
as well.

\begin{lem}\label{lem:redklassed}
$\{g\}_\phi k=\{g\,k\}_{\t_{k^{-1}}\circ\phi}$.
\end{lem}

\begin{proof}
Let $g'=f\,g\,\phi(f^{-1})$ be $\phi$-conjugate to $g$. Then
$$
g'\,k=f\,g\,\phi(f^{-1})\,k=f\,g\,k\,k^{-1}\,\phi(f^{-1})\,k
=f\,(g\,k)\,(\t_{k^{-1}}\circ\phi)(f^{-1}).
$$
Conversely, if $g'$ is $(\t_{k^{-1}}\circ\phi)$-conjugate to $g$, then
$$
g'\,k^{-1}=
f\,g\,(\t_{k^{-1}}\circ\phi)(f^{-1})k^{-1}=
f\,g\,k^{-1}\,\phi(f^{-1}).
$$
Hence a shift maps $\phi$-conjugacy classes onto classes related to
another automorphism.
\end{proof}

\begin{cor}\label{cor:rphiequaltaurphi}
$R(\phi)=R(\t_g \circ \phi)$.
\end{cor}

If $H\ss G$ is a $\phi$-invariant and normal, the quotient map
induces an epimorphism of Reidemeister classes, in particular,
\begin{equation}\label{eq:epiRade}
    R_G(\phi) \ge R_{G/H}(\ov{\phi})
\end{equation}
for the induced automorphism $\ov{\phi}:G/H \to G/H$ (see
\cite{goncalves,go:nil1,gowon}
and \cite[Sect. 3]{polyc}  for details and refinements).

\begin{dfn}\label{dfn:resfin}
A group $G$ is called \emph{residually finite}
if for any finite set $K\ss G$
there exists a normal group $H$ of finite index
such that $H\cap K=\varnothing$. Taking $K$ formed
by $g_i^{-1}g_j$ for some finite set $K_0=\{g_1,\dots,g_s\}$
one obtains an epimorphism $G\to G/H$ onto a finite group,
which is injective on $K_0$.
\end{dfn}

\begin{rk}\label{rk:fgrfmainprop}
If a residually finite group $G$ is finitely generated,
then groups $H$ in the definition
can be supposed to be characteristic
(by Lemma \ref{lem:normalcharfiniteind}).
\end{rk}

We will remind now some facts from Representation
Theory and Harmonic Analysis (see \cite{DixmierEng}
and \cite{BekkaHarpeValetteT} for an effective
introduction).
(Left) \emph{regular representation} $\l_G$ is the
unitary representation of $G$ on $\ell^2(G)$ by left
translations. The completion $C^*_\l(G)$
of $\ell^1(G)$ by the norm of $B(\ell^2(G)$
is called \emph{reduced group $C^*$-algebra} of $G$.
The completion $C^*(G)$
of $\ell^1(G)$ by the norm of all
unitary representations
is called \emph{(full) group $C^*$-algebra} of $G$.
The algebra $C^*_\l(G)$ is a quotient of $C^*(G)$.

Non-degenerate representations of $C^*(G)$ are exactly unitary
representations of $G$, in particular, $\wh {C^*(G)}=\wh G$. For a
representation $\r$ of $G$ we denote by $C^*\r$ the corresponding
representation of $C^*(G)$, and by $C^*\Ker \r$ the kernel of $C^*
\r$. One introduces on $\wh G$ the \emph{Jacobson-Fell} or
\emph{hull-kernel} topology defining the closure of a set $X$ by
the following formula
$$
\overline{X}=\{[\rho]\: :\: C^*\text{Ker} \rho \supseteq
\bigcup_{[\pi]\in X} C^*\text{Ker} \pi\}.
$$
This topology can be described in terms of
\emph{weak containment}: a representation $\r$ is
weakly contained in representation $\pi$
(we write $\r \prec\pi$)
if
diagonal matrix coefficients of $\r$ can be
approximated by linear
combinations of diagonal matrix coefficients of $\pi$
uniformly on finite sets. Here a \emph{matrix coefficient}
of a representation $\r$ on a Hilbert space $H$ is
the function $g\mapsto \<\r(g)\xi,\eta\>$ on $G$ for some
fixed $\xi,\eta\in H$, and a \emph{diagonal} one corresponds
to $\xi=\eta$. Then $C^*\Ker \pi \ss C^*\Ker \r$ if
and only if $\r \prec\pi$. Since
$$
C^*\Ker(\r_1\oplus\dots\oplus\r_m)=\cap_{i=1}^m C^*\Ker \r_i,
$$
\begin{equation}\label{eq:weakcontsum}
    \r_1\oplus\dots\oplus\r_m \prec \pi \mbox{ if } \r_i \prec \pi,
    \quad i=1,\dots,m.
\end{equation}

An \emph{amenable group} may be characterized in several
equivalent ways (see e.g. \cite{BekkaHarpeValetteT}), in
particular:
\begin{itemize}
    \item There exists an invariant mean on $\ell^\infty(G)$.
    \item $1_G \prec \l_G$, where $1_G$ is the trivial
    1-dimensional representation.
    \item $C^*(G)=C^*_\l(G)$.
\end{itemize}

A group $G$ has \emph{property (T)} if the trivial 1-dimensional
representation $1_G$ is an isolated point of $\wh G$ and is
$C^*$-\emph{simple} if $\rho \prec \l$ implies $\l\prec \rho$ for
any $\rho$ \cite{BekkaHarpeValetteT,Harpe2007simple}.

\begin{dfn}
A function $f:G\to\C$ is called \emph{positively definite}
function if for any finite collection $\{g_1,\dots,g_s\}\ss G$
the matrix $\|f(g_i^{-1}g_j)\|$ is non-negatively definite.
\end{dfn}

\begin{dfn}
\emph{Fourier-Stieltijes algebra} $B(G)$ can be defined in
three equivalent ways:
\begin{enumerate}
    \item $B(G)$ is formed by finite linear combinations
    of positively definite functions;
    \item $B(G)$ is formed by all matrix coefficients of
    all unitary representations of $G$;
    \item $B(G)=C^*(G)'$, i.e. it is the Banach space dual
    to $C^*(G)$.
\end{enumerate}
\end{dfn}

\begin{rk}\label{rk:separcoefffdim}
Matrix coefficients of distinct finite-dimensional representations
are linear independent (see \cite[Corollary (27.13)]{CurtisReiner}
for an algebraic argument). Another argument uses the fact that
these representations are disjoint having maximal ideals as
kernels. Passing to the weak containment interpretation gives the
linear independence.
\end{rk}

Let us remind now some facts from \cite{FelTro,polyc}.
\begin{lem}\label{lem:invarfunccodimetc}
Let $\rho$ be a finite dimensional irreducible representation
of $G$ on $V_\rho$, and $\phi:G\to G$ is an automorphism.

1). There exists a twisted invariant function $\w:G\to \C$ being
a matrix
coefficient of $\rho$ if and only if $\wh\phi[\rho]=[\rho]$.

2). In this case such $\w$ is unique up to scaling.

3). The space $K_\r:=\{b \in \End V_\rho\: |\: b=a-\r(g) a
\r(\phi(g^{-1})$ for some $g\in G$ and some $a  \in \End V_\rho\}$
has codimension 1 if $\wh\phi[\rho]=[\rho]$ and coincides with
$\End V_\rho$ otherwise.

4). If we have several distinct $\wh\phi$-fixed classes, then
the correspondent twisted invariant functions are independent.
\end{lem}

\begin{proof}
Let us sketch a proof, the details can be found in
\cite{FelTro,polyc,ncrmkwb}.
Matrix coefficients of a finite dimensional representation $\r$
arise from functionals on $\End V_\r$ and
can be written as $g\mapsto \Trace (a \r(g))$ for some
matrix $a\in\End V_\r$. Since the equality
$$
0=\Trace(ab)-\Trace(a \r(h) b \r(\phi(h^{-1})))=
\Trace ((a- \r(\phi(h^{-1}))a \r(h)) b)
$$
for any $b$ and $h$ implies $\r(\phi(h))a= a \r(h)$,
the above matrix coefficient is twisted invariant if and only if
$a$ is an intertwining operator between $\r$ and $\phi\circ\r$.
This gives 1), and Schur's lemma gives 2).
Now, 3) follows immediately from 1) and 2), because $K_\r$ is evidently
the intersection of kernels of all twisted invariant functionals on
$\End V_\rho$.
Finally, 4) follows from Remark \ref{rk:separcoefffdim}.
\end{proof}

\section{Fixed points and Reidemeister numbers}\label{sec:fixpoiReid}
\begin{teo}[an extraction from \cite{Jabara}]\label{teo:FixJabara}
Let $\G$ be a finitely generated residually finite group.
If $C_\G(\phi)=\Fix(\phi)$ is infinite, then $R(\phi)=\infty$.
\end{teo}

\begin{proof}
By Lemma \ref{lem:jabarafinoz} and inequality (\ref{eq:sravnnnilog})
for any automorphism $\psi$ of a finite
group $G$
\begin{equation}\label{eq:FixFinJabara}
    \sqrt{\log_2 |C_G(\psi)|}\le R_G(\psi).
\end{equation}
Let $\{x_1,x_2,\dots\}=C_\G (\phi)$. Then for every $n$ we can
find (see Remark \ref{rk:fgrfmainprop}) a characteristic subgroup
$\G_n$ of finite index in $\G$ such that the quotient map
$p_n:\G\to\G/\G_n=:G_n$ is injective on $\{x_1,\dots, x_n\}$. Let
$\phi_n:G_n\to G_n$ be the induced automorphism. Then $\{p_n
(x_1),\dots, p_n (x_n)\}\subset C_{G_n}(\phi_n)$, hence
(\ref{eq:epiRade}) and (\ref{eq:FixFinJabara}) imply
$$
R_\G(\phi)\ge R_{G_n}(\phi_n)\ge \sqrt {\log_2 |C_{G_n}(\phi_n)|}
\ge \sqrt {\log_2 n}.
$$
Since $n$ was arbitrary, we are done.
\end{proof}

\begin{cor}
Suppose, $\G$ is a torsion-free finitely generated residually finite group.
If $C_\G(\phi)>1$, then $R(\phi)=\infty$.
\end{cor}

\begin{proof}
Indeed, in this case $\G$ has an infinite subgroup$\cong\Z$ formed
by fixed points.
\end{proof}

\begin{teo}\label{teo:twistJabara}
Let $\G$ be as in Theorem \ref{teo:FixJabara} and $R(\phi)<\infty$.
Then $| St^{tw}_{\phi}(g, s g \phi(s^{-1})) | <\infty $ for any
$g,s \in \G$.
\end{teo}

\begin{proof}
By (\ref{eq:stabilcoset}) we need to estimate $ | St^{tw}_{\phi}(g)|$
only. Under the left translation by $g^{-1}$ the Reidemeister class
$\{g\}_\phi$ will be moved to the class $\{e\}_{\t_{g^-1}\circ\phi}$
(see Lemma \ref{lem:redklassed}). Thus
$$
|St^{tw}_{\phi}(g)|=|St^{tw}_{\t_{g^-1}\circ\phi}(e)|=
|C_\G(\t_{g^-1}\circ\phi)|.
$$
If it is infinite, by Theorem \ref{teo:FixJabara}
and Corollary \ref{cor:rphiequaltaurphi}
we have
$\infty = R(\t_{g^-1}\circ\phi)=R(\phi)$. A contradiction.
\end{proof}

\begin{cor}\label{cor:infiniteclasses}
Let $\G$ and $\phi$ be as in Theorem \ref{teo:twistJabara}
and $|\G|=\infty$. Then all Reidemeister classes of $\phi$
are infinite.
\end{cor}

\begin{proof}
Indeed, $|\G|=|\{g\}_\phi|\cdot | St^{tw}_{\phi}(g)|$.
\end{proof}

\section{Twisted Burnside-Frobenius theorem}\label{sec:TBFT}
\begin{teo}\label{teo:TBFTforRF}
Let $R(\phi)<\infty$ for an automorphism $\phi$ of a
finitely generated
residually finite group $G$.
Then $R(\phi)$ is equal to the number of finite dimensional 
fixed points of $\wh\phi$.
\end{teo}

\begin{proof}
$R(\phi)$ equals the dimension of the space of twisted invariant
elements of $\ell^\infty(G)$, i.e. functionals on $\ell^1(G)$ such
that their kernels contain the closure $K_1$ in $\ell^1(G)$ of the
space of elements of the form $b-g[b]$, $g[b](x):=b(g x
\phi(g^{-1})$.

Since $R(\phi)<\infty$, $\mathrm{codim}\, K_1 = R(\phi)$, and
$K_1$ has a Banach space complement of dimension $R(\phi)$. We can
take it in a way such that it has a base $a_i \in \C[G]$,
$i=1,\dots,R(\phi)$, i.e., all $a_i$'s have a finite support. Let
$p:G \to F=G/H$ be an epimorphism on a finite group $F$ such that
it distinguishes all elements from the union of (finite) supports
of $a_i$ and $H$ is characteristic (see Remark
\ref{rk:fgrfmainprop}). The image of $\ell^1(G)$ under the induced
homomorphism $p_1$ is $\ell^1(F)=\C[F]$. Also $K_1$ maps
epimorphically onto the space $K_p$ of elements
$\b-p(g)[\b]=p_1(b)-p(g)[p_1(b)]=p_1(b-g[b])$  in $\C[F]$. Thus,
$\{p_1(a_i)\}$ form a basis of a complement to $K_p$ in $\C[F]$.

Decompose this (finite dimensional)
algebra $\C[F]$ into a direct sum
of matrix algebras, i.e., decompose the image
 $p_1(\ell^1(G))$ (in fact, $p_1(C^*(G))$)
 into blocks:
$$
J:\C[F]\cong\oplus_{i=1}^N \End V_i,
$$
where $\r_i$ ($i=1,\dots,N$) are some irreducible representations
on $V_i$.
Since this decomposition is in fact associated with
the decomposition of the left regular
representation of
$$
\l_F\cong\oplus_{i=1}^N V_i\otimes V_i^*,
$$
these $\r_i$'s are
pairwise non-equivalent \cite[I, \S 2]{serrerepr}.
Let $K_i$ be formed by $x-\r_i(g)[x]$ in $\End V_i$.
Since $J$ is an algebra isomorphism,
$$
R(\phi)=\mathrm{codim}\, K_1= \sum_i \mathrm{codim}\, K_i.
$$
The last one is 1 if $\wh \phi (\r_i)= \r_i$ (in this
case  intertwining operators give a linear complement to $K_i$)
and 0 otherwise by Lemma \ref{lem:invarfunccodimetc}.
Thus, $R(\phi)\le$ the number of finite dimensional 
fixed points of $\wh\phi$.

The inverse estimation follows from Lemma \ref{lem:invarfunccodimetc}.
\end{proof}

\begin{rk}
We use here an approach from
\cite{ncrmkwb}) with $\ell^1(G)$ instead of $C^*(G)$.
\end{rk}

\begin{rk}
As it is known (see e.g. \cite{Remes,Wehrfritz73JLMS})
f.g. residually finite groups
($e$ is closed in the profinite topology) form a strictly
larger class then conjugacy separable groups (each conjugacy
class is closed). Our argument is not applicable to
a ``proof'' of a wrong fact (coincidence of these classes).
Indeed, we have no map to identify invariant elements
from $\ell^\infty(G)$ with a complement to $K_1$. So,
our argument uses in a crucial way the finiteness of $R(\phi)$,
while the number of usual conjugacy classes is typically
infinite. We have an infinite number of independent
characters of finite dimensional irreducible representations
for a residually finite group and can separate infinitely
many ordinary conjugacy classes, but not necessary all of them.
\end{rk}

\begin{cor}
In the situation of Theorem \ref{teo:TBFTforRF}
$\phi$-class functions are in $B(G)$, Fourier-Stieltijes
algebra of $G$ $($c.f. the discussion in {\rm \cite{ncrmkwb})}.
\end{cor}

A group $G$ is $\phi$-conjugacy separable, if there is a
homomorphism onto finite group $G/H$ with a $\phi$-invariant
normal subgroup $H$ inducing a bijection of twisted conjugacy
classes.

\begin{cor}\label{cor:phiconjRF}
In the situation of Theorem \ref{teo:TBFTforRF}
$G$ is $\phi$-conjugacy separable.
\end{cor}

One can slightly modify the proof of Theorem \ref{teo:TBFTforRF}
to obtain the following statement.

\begin{prop}
Suppose, $R(\phi)=\infty$  for an automorphism $\phi$ of a
finitely generated residually finite group $G$. Then the number of
finite dimensional fixed points of $\wh\phi$ is infinite.
\end{prop}

\begin{proof}
Let $M$ be an arbitrary positive integer. Consider a subspace
$L\ss \ell^1(G)$ such that $K_1 \ss L$ and $\mathrm{codim} (L)=m$.
Choose a complement to $L$ with a base $\{a_1,\dots,a_m\}$ of
finitely supported functions and corresponding $F$ and $H$ as in
the proof of Theorem \ref{teo:TBFTforRF}. In the same way $K_1$
maps onto $K_p$ and the images of $a_i$ form a part of a basis of
a complement to $K_p$ in $\C[F]$. We obtain $m\le$ the number of
finite dimensional fixed points of $\wh\phi$. Since $m$ is
arbitrary, we obtain the statement.
\end{proof}

\section{Amenability, twisted inner representation,
and $R_\infty$ property}\label{sec:amentwistinnrinfty}

Our argument here partially follows \cite{KaniuMark1992}.
\begin{dfn}
Denote by $\g^\phi_G$ the
\emph{twisted inner representation} of $G$
on $\ell^2(G)$, i.e.
$$
\g^\phi_G(x)(f)(g)=f(xg\phi(x^{-1})),\quad x,g\in G,\quad
f\in\ell^2(G).
$$
\end{dfn}
Denote $C_\phi(a):=St^{tw}_\phi(a)$, $a\in G$.
Evidently, $\g^\phi_G$
decomposes into a direct sum of representations
 $\g^\phi_a$ being restrictions of $\g^\phi_G$
 onto $\{a\}_\phi$
(i.e. on $\ell^2(\{a\}_\phi)$).

\begin{lem}\label{lem:ogranichgammaphi}
The representation $\g^\phi_a$ is equivalent to the induced
representation $\ind ^G_{C_\phi(a)} 1_{C_\phi(a)}$.
\end{lem}

\begin{proof}
Indeed, this induced representation $T$ can be realized
on $\ell^2(C_\phi(a)\setminus G)$
by the following action
$[T(g)(f)](x)=f(xg)$, $x\in C_\phi(a)\setminus G$, $g\in G$,
where $C_\phi(a)\setminus G$ is the space of left cosets
by $C_\phi(a)$.
Let us identify $C_\phi(a)\setminus G$ with $\{a\}_\phi$
by $i(C_\phi(a)\cdot g)=\g^\phi_G(g)(a)$. Evidently, this
map is well defined and gives a unitary isomorphism
$$
I: \ell^2(\{a\}_\phi) \to \ell^2(C_\phi(a)\setminus G),
\quad I(f)(x):=f(i(x)).
$$
Then
$$
[I \circ \g^\phi_G (g)(f)](x)=
[ \g^\phi_G (g)(f)](i(x))= f(g h a \phi((gh)^{-1})),
\quad x= C_\phi(a)\cdot h,
$$
$$
[T(g)\circ I (f)](x)=  I(f)(xg)= f(i (xg))=
f(\g^\phi_G(hg)(a))=f(g h a \phi((gh)^{-1})).
$$
Thus, $I$ is an intertwining unitary.
\end{proof}

\begin{teo}\label{teo:weakecontgammaphiregul}
Suppose,   $|C_\phi(a)|<\infty$,
for any $a\in G$. Then $\g^\phi_G$ is weakly
contained in the regular representation $\l_G$.
\end{teo}

\begin{proof}
The characteristic functions $\chi_{C_\phi(a)}$, $a\in G$,
are positively definite functions associated to $\l_G$,
because they are finite sums of translations of $\d_e$.
Hence, $\ind ^G_{C_\phi(a)} 1_{C_\phi(a)}\prec \l_G$
(cf. \cite[E.4.4]{BekkaHarpeValetteT}).
By Lemma \ref{lem:ogranichgammaphi} and the decomposition
of $\g^\phi_G$ we obtain $\g^\phi_G \prec \l_G$.
\end{proof}

\begin{teo}\label{teo:rfinimpliesamenab}
Suppose, $G$ is a finitely generated residually finite group and
$R(\phi)<\infty$ for some $\phi:G\to G$. Then $G$ is amenable.
\end{teo}

\begin{proof}
In this case (see Theorem \ref{teo:twistJabara} above) $|C_\phi(a)|<\infty$,
for any $a\in G$. Thus, by Theorem \ref{teo:weakecontgammaphiregul}
$\g^\phi_G \prec \l_G$. So, it is sufficient to
verify that $1_G \prec \g^\phi_G$, i.e.
$C^*\Ker \g^\phi_G \ss  C^*\Ker 1_G$.

Suppose, $f=\sum_{g\in G} f^g \d_g \in C^*\Ker \g^\phi_G$,
i.e.
\begin{equation}\label{eq:kernelgammaphi}
\sum_{g\in G} f^g \g^\phi_G(g) \d_h =
\sum_{g\in G} f^g L_g R_{\phi(g^{-1})} \d_h =
\sum_{g\in G} f^g \d_{gh\phi(g^{-1})}=0\qquad
\hbox{ for any } h\in G.
\end{equation}
Since $gh\phi(g^{-1})=g_1 h\phi(g_1^{-1})$ if and
only if $g_1,g_2 \in St^{tw}_\phi(h,x)$ for some
$x\in \{h\}_\phi$,
(\ref{eq:kernelgammaphi}) is equivalent to
\begin{equation}\label{eq:kernelgammaphi1}
\sum_{g\in St^{tw}_\phi(h,x)} f^g =0, \quad
\hbox{ for any } h\in G \hbox{ and }x\in \{h\}_\phi.
\end{equation}
Take any $h$, e.g. $h=e$. Then $St^{tw}_\phi(e,e)=C_\phi(e)$
is a finite subgroup of $G$ and the other $St^{tw}_\phi(h,x)$
form the set of all cosets with the respect to this subgroup.
This implies $\sum_{g\in G} f^g=0$ (for any reasonable
convergence, because the subgroup is finite). In particular,
$f\in C^*\Ker 1_G$ and $1_G \prec \g^\phi_G \prec \l_G$.
\end{proof}

\begin{rk}
Main results of \cite{Jabara} illustrate this theorem.
\end{rk}

\begin{cor}\label{cor:mainrinf}
Any non-amenable residually finite finitely generated group is an
$R_\infty$ group.
\end{cor}

\begin{rk}\label{rk:onlufinit}
In fact we have used only the estimation $|C_\phi(a)|<\infty$, for
any $a\in G$. This can occur not only when the group is residually
finite. The following case is especially important for
applications.
\end{rk}

Suppose, we have an exact sequence
\begin{equation}\label{eq:exactckonjad}
    0\to F \to G \ola{p} H=G/F \to 0,\qquad |F|=m <\infty,
\end{equation}
$G$ is finitely generated, $H$ is residually finite and
non-amenable, $F$ not necessary characteristic, in particular,
$\phi$-invariant (otherwise we have the $R_\infty$ property with
the help of epimorphity of mapping of Reidemeister classes). As in
Theorem \ref{teo:FixJabara}, suppose $\{x_1,x_2,\dots\}=C_G
(\phi)$ is infinite. Let $H_n$ be a characteristic subgroup of
finite index in $H$ such that the quotient map $p_n:H\to
H/H_n=:F_n$ is injective on $\{p(x_1),\dots, p(x_n)\}$. Then
$G_n:=\Ker (p_n\circ p)$ is a subgroup of $G$ of index $|F_n|$. As
above, we can assume $G_n$ to be characteristic, since $G$ is
finitely generated. Denote $s:= \# \{p(x_1),\dots, p(x_n)\}$.
Thus, $s\ge n/m$. Then
$$
R_G(\phi) \ge R_{F_n} (\phi_n) \ge \sqrt {\log_2
|C_{F_n}(\phi_n)|} \ge \sqrt {\log_2 s} \ge \sqrt {\log_2 n-\log_2
m}.
$$
Since $n$ was arbitrary, we have
$R_G(\phi)=\infty$. Thus, if $R_G(\phi)<\infty$, then
$C_G(\phi)<\infty$. Hence,
$C_\phi(a)<\infty$, and by Remark \ref{rk:onlufinit} we obtain the
following statement.
\begin{prop}\label{prop:usilennonamen}
Suppose, we have $($\ref{eq:exactckonjad}$)$, where $G$ is
finitely generated, $H$ is residually finite and non-amenable,
$|F|<\infty$. Then $G$ is an $R_\infty$-group.
\end{prop}

Another large class of examples arises from
\begin{prop}\label{prop:rinftysubgrcsimple}
Let $G$ be a normal finitely generated subgroup of a residually
finite $C^*$-simple group. Then $G$ is an $R_\infty$-group.
\end{prop}

\begin{proof}
Since $G$ is non-amenable (see e.g. \cite[Prop.
3]{Harpe2007simple}), $G$ enter conditions of Corollary
\ref{cor:mainrinf}.
\end{proof}

\section{Rational points}\label{sec:rationalpoints}
In this section we show that not every finite-dimensional
representation can be fixed by $\wh\phi$ if $R(\phi)<\infty$.
\begin{dfn}
Let $[\r]\in \wh G_f$, $g\mapsto T_g$ be a (class of a) finite-dimensional
representation. We say that $\r$ is \emph{rational} if the number
of distinct $T_g$'s is finite, and \emph{irrational} otherwise.
\end{dfn}

\begin{rk}
Evidently, $\r$ is rational if and only if it can be factorized
through a homomorphism $G\to F$ on a finite group.
\end{rk}

As a consequence of the proof of Theorem \ref{teo:TBFTforRF}
(see also Corollary \ref{cor:phiconjRF})
we obtain the following statement.

\begin{teo}
Suppose, $G$ is a finitely generated
residually finite group, and $\phi$ is its automorphism
with $R(\phi)<\infty$. Then no irrational representation is fixed by
$\wh\phi$.
\end{teo}

\begin{teo}
Let $G$ be a finitely generated group and for an automorphism $\phi$
at least one of the following two conditions holds:

\begin{enumerate}[\rm 1).]
    \item There exist infinitely many finite-dimensional
    representation classes in $\wh G$ fixed by $\wh\phi$.
    \item There exists an irrational representation $\rho$ fixed
    by $\wh\phi$.
\end{enumerate}

Then $R(\phi)=\infty$.

In particular, if we have one of these conditions for
every automorphism $\phi$, then $G$ has $R_\infty$
property.
\end{teo}

\begin{proof}
1) This follows immediately from Lemma \ref{lem:invarfunccodimetc}.

2) Suppose that $f_\r(g)=\Trace(a\r(g))$
(see Lemma \ref{lem:invarfunccodimetc}) takes
finitely many values and will arrive to a contradiction.
Indeed, $f_\r$ is a non-trivial matrix coefficient. Hence,
(see, e.g. \cite{Kirillov}) its left translations generate
a finite-dimensional representation, which is equivalent
to a direct sum of several copies of $\r$. The space $W$
of this representation has a basis $L_{g_1} f_\r,\dots,L_{g_k} f_\r$.
Thus, all functions from $W$ take only finitely many
values (with level sets of the form $\cap_i g_i U_j$, where
$U_j$ are the level sets of $f_\r$). Taking unions of these
sets (if necessary) we can form a finite partition
$G=V_1\sqcup\dots\sqcup V_m$ such that elements of $W$ are
constant on the elements of the partition and for each
pair $V_i \ne V_j$ there exists a function from $W$ taking distinct
values on them. Thus any left translation maps $V_i$ onto
each other and the representation $G$ on $W$ factorizes
through (a subgroup of) the permutation group on $m$
elements, i.e. a finite group. The same is true for its
subrepresentation $\r$, thus it is rational.
A contradiction.\footnote{Probably this argument
is known, but we have not found an appropriate
reference.}
\end{proof}


\section{Isogredience classes}\label{sec:isogred}


\begin{dfn}\label{dfn:isogred}(see \cite{ll})
Suppose, $\F \in \Out (G):= \Aut (G)/ \Inn (G)$. We say that
$\a,\b\in \F$ are \emph{isogredient} (or \emph{similar})
if $\b=\t_h \circ \a \circ \t_h^{-1}$ for some $h\in G$.

Let $\mathcal S(\F)$ be the set of isogredience classes of $\F$.
If $\F=\Id_G$, then above $\a$ and $\b$ are inner, say
$\a = \t_g$, $\b=\t_s$. Since elements of center $Z(G)$ give
trivial inner automorphisms, we may suppose $g,s\in G/Z(G)$.
Then the equivalence relation takes the form
$\t_s=\t_{h g h^{-1}}$, i.e., $s$ and $g$ are conjugate
in $G/Z(G)$. Thus, $\mathcal S(\Id)$ is the set of conjugacy
classes of $G/Z(G)$.

Denote by $S(\F)$ the cardinality of $\mathcal S(\F)$.
\end{dfn}

For a topological motivation of the above
definition of the isogredience
suppose that $G \cong \pi_1(X)$, $X$ is a compact space, and
$\Phi$ is induced  by a continuous map   $f: X \rightarrow X $.
Let   $p:\widetilde{X}\rightarrow X$ be the
universal covering of $X$ ,
$\widetilde{f}:\widetilde{X}\rightarrow \widetilde{X}$ a
lifting of $f$, i.e. $p\circ\widetilde{f}=f\circ p$.
 Two liftings $\tilde{f}$ and
$\tilde{f}^\prime$ are called
{\it isogredient} or conjugate if there is a
$\gamma\in G $ such that $\tilde{f}^\prime =
\gamma\circ\tilde{f}\circ\gamma^{-1}$.
Different  lifting  may have very different properties.
 Nielsen observed \cite{nie1}
 that conjugate lifting of homeomorphism of
 surface  have similar dynamical properties.
 This led him to the definition of the
 isogredience of liftings in this case.
 Later Reidemeister \cite{reid:re} and
 Wecken \cite{Weck}
 succeeded in generalizing the theory to
 continuous maps of compact polyhedra.

 The set of isogredience classes of
 automorphisms representing a
 given outer automorphism and the
 notion of index $\Ind(\Phi)$
 defined via the set of isogredience
 classes are strongly related to important
 structural properties of $\Phi$
 (see \cite{GabJaegLevittLustig98Duke}),
 for example in another (with respect to Bestvina--Handel
 \cite{BestvinaHandel92})
 proof of the Scott conjecture \cite{GabLeviLus98}.

One of the main results of \cite{ll} is that for any
non-elementary hyperbolic group and any $\F$ the set
$\mathcal S(\F)$ is infinite, i.e.,
$S(\F)=\infty$. We will extend this result.
First, we introduce an appropriate definition.

\begin{dfn}\label{dfn:Sinfin}
A group $G$ is an $S_\infty$-\emph{group} if for
any $\F$ the set $\mathcal S(\F)$ is infinite,
i.e., $S(\F)=\infty$.
\end{dfn}
Thus, the above result from \cite{ll} says:
any non-elementary hyperbolic group is an $S_\infty$-group.
On the other hand, finite and Abelian groups are
evidently non $S_\infty$-groups.

Now, let us generalize the above calculation
for $\F=\Id$ to a general $\F$  (see \cite[p. 512]{ll}).
Two representatives of $\F$ have form $\t_s\circ \a$,
$\t_q\circ  \a$,
with some $s,q \in G$. They are isogredient if and only
if
$$
\t_q \circ \a =\t_g\circ  \t_s\circ  \a\circ \t_g^{-1}
=\t_g\circ  \t_s\circ  \t_{\a(g^{-1})} \circ\a,
$$
$$
\t_q=\t_{gs\a(g^{-1})},\qquad q=gs\a(g^{-1})c,\quad c\in
Z(G).
$$
So, the following statement is proved.
\begin{lem}\label{lem:SandR}
$S(\F)=R_{G/Z(G)}(\ov\a)$, where $\a$ is any representative
of $\F$ and $\ov\a$ is induced by $\a$ on $G/Z(G)$.
\end{lem}

Since $Z(G)$ is a characteristic subgroup, we obtain
from Lemma \ref{lem:SandR} and epimorphity (see before
(\ref{eq:epiRade})) the following statement
(in one direction it was discussed in
\cite[Remark 2.1]{gowon}).

\begin{teo}\label{teo:SandRfinZ}
Suppose,
$|Z(G)|<\infty$.
Then $G$ is an $R_\infty$-group
if and only if
$G$ is an $S_\infty$-group.
\end{teo}

\begin{rk}
Of course, this argument is applicable to an
individual $\F$ as well.
\end{rk}

Theorem \ref{teo:SandRfinZ} and Corollary
\ref{cor:mainrinf} imply

\begin{cor}\label{cor:isogrnonamenfincen}
Suppose, $G$ is a finitely generated non-amenable
residually finite group with $|Z(G)|<\infty$.
Then $G$ is an $S_\infty$-group.
\end{cor}

\begin{cor}\label{cor:isogrsimple}
Suppose, $G$ is a finitely generated $C^*$-simple
residually finite group.
Then $G$ is an $S_\infty$-group.
\end{cor}

\begin{proof}
In this case the center is trivial and
the group is non-amenable (see e.g.
\cite{Harpe2007simple}).
\end{proof}

Now we can give a more advanced example
of a non $S_\infty$-group. Namely, consider
Osin's group \cite{Osin}. This is a non-residually
finite exponential growth group with two conjugacy
classes. Since it is simple, it is not $S_\infty$
by Theorem \ref{teo:SandRfinZ}.

 \section{Examples}\label{sec:AppList}

\subsection{Known examples of $R_\infty$ groups}
It was shown by various authors that the following groups
have the $R_\infty$-property:

\begin{enumerate}
\item\label{it:gromov} non-elementary Gromov hyperbolic groups
\cite{FelPOMI,ll}; relatively hyperbolic groups \cite{f07};

\item \label{it:BS} Baumslag-Solitar groups $BS(m,n)$
except for $BS(1,1)$ \cite{FelGon08Progress}, generalized
Baumslag-Solitar groups, that is, finitely generated groups which
act on a tree with all edge and vertex stabilizers infinite cyclic
\cite{LevittBaums}; the solvable generalization $\G$ of $BS(1,n)$
given by the short exact sequence $1 \rightarrow \mathbb
Z[\frac{1}{n}] \rightarrow \G \rightarrow \mathbb Z^k \rightarrow
1,$ as well as any group quasi-isometric to $\G$ \cite{TabWong};

\item \label{it:SWB}
a wide class of saturated weakly branch groups (including the
Grigorchuk group \cite{GrFA} and the Gupta-Sidki group
\cite{GuSi}) \cite{FelLeonTro}, Thompson's group $F$ \cite{bfg};
generalized Thompson's groups $F_{n,\:0}$ and their finite direct
products \cite{GK};

\item \label{it:braids} symplectic groups $Sp(2n,\mathbb Z)$,
the mapping class groups $Mod_{S}$ of a compact oriented surface
$S$ with genus $g$ and $p$ boundary components, $3g+p-4>0$, and
the full braid groups $B_n(S)$ on $n>3$ strings of a compact
surface $S$ in the cases where $S$ is either the compact disk $D$,
or the sphere $S^2$ \cite{dfg}, some classes of Artin groups of
infinite type \cite{juhasz};

\item\label{lingrext}   extensions
of  $SL(n,\Z)$, $PSL(n,\Z)$, $GL(n,\Z)$, $PGL(n,\Z)$, $Sp(2n,\Z)$,
$PSp(2n,\Z)$, $n>1$, by a countable abelian group, and normal
subgroup of $SL(n,\Z)$, $n>2$, not contained in the centre
 \cite{MubeenaSankaran1111.6181};

\item\label{it:linN}  $GL(n,K)$
and $SL(n,K)$ if $n>2$ and
$K$ is an infinite integral domain
with trivial group of automorphisms, or
$K$ is an integral domain, which has
a zero characteristic and for which $\Aut(K)$ is torsion
\cite{Nasybullov1201.6515};

\item\label{it:semisla}
irreducible lattice in a connected semi
simple Lie group $G$ with finite center and real rank
at least 2 \cite{MubeenaSankaran1201.4934};

\item\label{it:somemetab}  some metabelian
groups of the form ${\mathbb Q}^n\rtimes \mathbb Z$
and ${\mathbb Z[1/p]}^n\rtimes \mathbb Z$ \cite{FelGon2011Q};
lamplighter groups
$\mathbb Z_n \wr \mathbb Z$ if and only if $2|n$ or
$3|n$ \cite{gowon1};
 free nilpotent group  $N_{rc}$ of rank $r=2$ and class $c\ge 9$
 \cite{GoWon09Crelle}, $N_{rc}$ of rank
$r = 2$ or $r = 3$ and class $c \geq 4r,$ or rank $r \geq 4$ and
class $c \geq 2r,$ any group $N_{2c}$ for $c \geq 4$, every free
solvable group $S_{2t}$ of rank 2 and class $t \geq 2$ (in
particular the free metabelian group $M_2 = S_{22}$ of rank 2),
any free solvable group $S_{rt}$ of rank $r \geq 2$ and class $t$
big enough \cite{Romankov}; some crystallographic groups
\cite{DePe2011,LutScze2011}.

\end{enumerate}

\subsection{Coverage by Theorem B}
Coverage of items (\ref{it:gromov}) and (\ref{it:semisla}) is
discussed in Introduction. Non-elementary (residually finite)
relatively hyperbolic without finite normal subgroups (in
particular, hyperbolic without torsion) groups are $C^*$-simple
\cite{ArzhanMinas2007JFA,Harpe88CRAS}. Thus their f.g. normal
subgroups are $R_\infty$ groups by Proposition
\ref{prop:rinftysubgrcsimple}.

This observation can be extended to the following large class
introduced in \cite{DGO}. Suppose that a residually finite
finitely generated group $G$ contains a non-degenerate
hyperbolically embedded subgroup \cite{DGO}. Then it is
non-amenable \cite[Theorem 8.5(b)]{DGO}, hence, $R_\infty$ by
Theorem B.\footnote{D.Osin has communicated to us that he has
another proof that groups with non-degenerate hyperbolically
embedded subgroup are $R_\infty$ groups.} If $G$ enters conditions
of \cite[Therem 8.11]{DGO}, for example, if $G$ has no nontrivial
finite normal subgroups, the $R_\infty$ property holds for any its
finitely generated residually finite subgroup by Proposition
\ref{prop:rinftysubgrcsimple}.

 Items (\ref{it:BS}), (\ref{it:SWB}),
(\ref{it:somemetab})
 are not
covered.

Most part of groups from item (\ref{lingrext}) is covered
by Theorem B.
Linear groups are residually finite (using $\Z\to \Z/k\Z$
on elements). Taking the quotient by the center commutes
with this map. Thus, projective linear groups are residually
finite as well. They are finitely generated and non-amenable.
Indeed, all of them have the free group $F_2$ as a subgroup
(for linear groups this is a direct fact
and the center evidently does not meet this $F_2$, hence
for the projective as well).

Groups from (\ref{it:braids}) are covered by Theorem B. For the
symplectic group it is explained above. Braid groups
$B_n=B_n(D^2)$ are residually finite finitely generated. Indeed,
one can use either Artin's representation in $\Aut(F_n)$ (see
\cite[2.19--2.20]{Paris}), or to use the exact sequence $ B_n
\leftarrow N_{1,n} \leftarrow \dots \leftarrow N_{n,n}=e
$
with all quotients being free groups except of the first one which
is the (finite) symmetric group (see, e.g. \cite{Markov45}). This
sequence implies non-amenability of $B_n$ for $n\ge 3$.

Mapping class groups for compact oriented surfaces are residually
finite \cite{Grossman}  and  are non-amenable for $3g+p-4>0$.
Indeed, these groups act on hyperbolic space
\cite{MasurMinsky,Bowditch} in such a way that
\cite{BestvinaFujiwara} is applicable. Some other cases can be
added after an individual analysis. The case of $B_n(S^2)$ follows
from Proposition \ref{prop:usilennonamen} applied to
$$
1\to \Z_2 \to B_n(S^2) \to Mod_{S^2_n}\to 1,
$$
where $S^2_n$ is the sphere with $n$ boundary components (see
\cite{Scott70}).

Under the supposition of f.g. the most part of groups from
(\ref{it:linN}) is covered by Theorem B being residually finite
(by Mal'cev's theorem) and containing $F_2\ss SL(2,\Z)$ for
characteristic 0. For a finite characteristic one can find an
embedded free product $\Z_m * \Z_n$ (in the cases covered by
\cite[Theorem 10.3]{Cohn1966IHES}). For characteristic $\ge 3$ one
has $m-1\ge 2$, $n-1\ge 2$, and $\Z_m * \Z_n$ is a Powers group,
in particular, non-amenable (see \cite[Corollary
12]{Harpe2007simple}).

The new classes, related to this item are pure braid groups (as it
is explained in Introduction) and (full) mapping class groups for
compact non-orientable surfaces (with some low-genus exclusions).
Indeed, let $N$ be a compact non-orientable surface, such that its
orienting cover $S$ is neither a sphere with $\le 4$ boundary
circles, nor a torus with $\le 2$ boundary circles, nor a closed
surface of genus $2$. Then $Mod_S$ is $C^*$-simple
\cite{BridsonHarpe2004JFA} and residually finite. Hence, $Mod_N$
is non-amenable \cite[Prop. 3]{Harpe2007simple} and residually
finite being a subgroup of $Mod_S$ (see, e.g. \cite[Sect.
2.1]{Atalan2010}).

Another new class is braid groups $B_n(\Sigma)$ and pure braid
groups $P_n(\Sigma)$ of a compact connected oriented surface
$\Sigma$ (for $n\ge 2$). Remind, that $P_n(\Sigma)$ is defined as
the fundamental group of $
F_n(\Sigma):=\{(x_1,\dots,x_n)\in\Sigma^n\: |\: x_i \ne x_j \mbox{
if } i\ne j\}. $ The symmetric group $S_n$ acts on $F_n(\Sigma)$
by permutation of coordinates, and the fundamental group of the
orbit space is the braid group $B_n(\Sigma)$. We obtain a regular
$n!$ fold covering and the exact sequence
$$
1\to P_n(\Sigma) \to B_n(\Sigma) \to S_n \to 1.
$$
Thus, Theorem B is applicable to $B_n(\Sigma)$ if and only if it
is applicable to $P_n(\Sigma)$. We have (\cite{Birman69}, see also
\cite{ParisRolfsen99}) an epimorphism $P_m(\Sigma)\to
\pi_1(\Sigma)^m$. If $\Sigma$ has genus $\ge 1$ and a border, then
$\pi_1(\Sigma)$ is free non-abelian (homotopicaly 1-dimensional
complex with at least 2 loops). Thus $P_n(\Sigma)$ is non-amenable
in this case. For a closed surface (still different from $S^2$ and
$\R P^2$) one can consider the short exact sequence from
\cite{FadellNeuwirth62}
$$
1\to \pi_1(\Sigma\setminus \{x_1,\dots,x_{n-1}\},x_n)\to
P_n(\Sigma)\to P_{n-1}(\Sigma)\to 1
$$
to prove non-amenability for $n\ge 3$ (now a free group is a
subgroup). Non-amenability, as well as the residual finiteness in
the case of a closed surface can be obtained from \cite[Theorem
5]{GoncGuasc034} (the residual finiteness can be also obtained by
\cite{BardBering}). The residual finiteness for surfaces with
boundary follows inductively from  \cite[Theorem 6]{GoncGuasc034}
and \cite[Theorem 7(b)]{mille-71}.  The several remaining cases
should be treated individually.

Artin's groups $A_\G$ (for a graph $\G$) are naturally included in
the exact sequence
$$
0\to PA_\G \to A_\G \to W_\G\to 0,
$$
where $PA_\G$ is the corresponding pure Artin's group and $W_\G$
is the corresponding Coxeter group. Evidently, $A_\G$ is
non-amenable (containing the free group). If it is residually
finite (e.g. if it is of finite (or spherical) type (i.e.
$|W_\G|<\infty$), or $W_\G$ is free abelian, see \cite[Theorem
7(c)]{mille-71}), we can apply Theorem B. Also, $A_\G$ has the
$R_\infty$ property, if $PA_\G$ is a characteristic subgroup and
$W_\G$ is neither finite nor almost abelian, because (as we have
proved above) in this case $W_\G$ has the $R_\infty$ property.
Actually, it is proved that $PA_\G$ is characteristic for the
finite type case \cite{NunoParis2006}. To compare this with
\cite{juhasz} is a separate problem.

\subsection{Known examples of $R(\phi)<\infty$,
Reidemeister spectrum, and Theorem A} Now we would like to list
some cases, where Theorem A is applicable.

Following \cite{Romankov}, define the {\it Reidemeister spectrum
of} $G$ as
$$
Spec(G):=\{k\in \N \cup \infty \: | \: R(\phi)=k \mbox{ for some }
\phi \in \Aut(G)\}.
$$
In particular, $R_\infty \Leftrightarrow Spec(G)=\{\infty\}$.

 It is easy to see that
  $Spec({\mathbb Z}) = \{2\} \cup \{ \infty \}$,
and,  for $n \geq 2$, the spectrum is full, i.e. $ Spec({\mathbb
Z}^n) = {\mathbb N} \cup \{\infty \}$. For free nilpotent groups
we have the following: $Spec(N_{22}) = 2{\mathbb N} \cup \{\infty
\}$ ($N_{22}$ is the discrete Heisenberg group)
\cite{Indukaev,FelIndTro,Romankov},  $ Spec(N_{23}) = \{ 2k^2 | k
\in {\mathbb N}\} \cup \{\infty \}$ \cite{Romankov} and $ Spec
(N_{32}) = \{2n - 1| n \in {\mathbf N}\} \cup \{ 4n | n \in
{\mathbf N}\} \cup \{ \infty \}$ \cite{Romankov}.

 Metabelian non-polycyclic groups have quite interesting
 Reidemeister spectrum \cite{FelGon2011Q}:
  for example, if  $\theta(1)$ is of the form
$
\begin{pmatrix}
r& 0 \\
0 & s
\end{pmatrix}
,$ then we have the following cases:

a) If  $r=s=\pm 1$ then $Spec(\mathbb Z[1/p]^2\rtimes_{\theta}
\mathbb Z)=\{2n |  n\in \mathbb N, \ (n,p)=1    \}\cup \{\infty\}$
where $(n,p)$   denote the greatest common divisor of $n$ and $p$.

b) If $r=-s=\pm 1$ then  $Spec( \mathbb
Z[1/p]^2\rtimes_{\theta}\mathbb Z)=\{ 2p^l(p^k\pm 1), 4p^l | l,
k>0\} \cup\{  \infty\}$.

c) If  $rs=1$ and $|r|\ne 1$ then   $Spec(\mathbb
Z[1/p]^2\rtimes_{\theta}\mathbb Z)=\{ 2(p^l \pm 1), 4 |l>0 \}
\cup\{ \infty \}$.

d) If either $r$ or $s$
does not have module equal to one,
and $ rs \ne 1$ then
  $Spec(\mathbb Z[1/p]^2\rtimes_{\theta}\mathbb
Z)=\{\infty\}$.

An interesting case was studied in \cite{FelTroVer}, where the
Reidemeister number and the number of fixed points of $\wh\phi$
where compared directly. In this example $G$ is a semidirect
product of $\Z^2$ and $\Z$ by Anosov automorphism defined by the
matrix $\scriptsize\Mat 2111$. $G$ is solvable and of exponential
growth. The automorphism $\phi$ is defined by $\scriptsize\Mat
01{-1}0$ on $\Z^2$ and as $-\Id$ on $\Z$.

Now we will show examples of some groups with extreme properties,
where $R(\phi)<\infty$ but Theorem A is not only unapplicable, but
the statement is wrong. Namely, Osin's group \cite{Osin} has
exactly 2 ordinary conjugacy classes, thus it is not residually
finite. Hence, by Mal'cev's theorem it is not linear. Thus, it has
no finite-dimensional representations without kernel. Since it is
simple, there is no representations with non-trivial kernel (only
the trivial 1-dimensional representation). Thus, $R(\text{Id})=2$,
but the (fixed) finite-dimensional representation is unique. A
similar argument works for Ivanov's group \cite[Theorem
41.2]{olsh} and some HNN-extensions \cite{HNN,serrtrees} (see
\cite{feltroKT} for details).


\begin{thebibliography}{100}

\bibitem{Arthur}
{\sc J.~Arthur and L.~Clozel}.
\newblock {\em Simple algebras, base change, and the advanced theory of the
  trace formula}.
\newblock Princeton University Press, Princeton, NJ, 1989.

\bibitem{ArzhanMinas2007JFA}
{\sc G.~Arzhantseva and A.~Minasyan}.
\newblock Relatively hyperbolic groups are {$C^\ast$}-simple.
\newblock {\em J. Funct. Anal.} {\bf 243},  No.~1, 345--351, 2007.

\bibitem{Atalan2010}
{\sc Ferihe Atalan}.
\newblock Outer automorphisms of mapping class groups of nonorientable
  surfaces.
\newblock {\em Internat. J. Algebra Comput.} {\bf 20},  No.~3, 437--456, 2010.

\bibitem{BardBering}
{\sc Valerij~G. Bardakov and Paolo Bellingeri}.
\newblock On residual properties of pure braid groups of closed surfaces.
\newblock {\em Comm. Algebra} {\bf 37},  No.~5, 1481--1490, 2009.

\bibitem{BassLubotzky83IsraelJM}
{\sc Hyman Bass and Alexander Lubotzky}.
\newblock Automorphisms of groups and of schemes of finite type.
\newblock {\em Israel J. Math.} {\bf 44},  No.~1, 1--22, 1983.

\bibitem{BekkaHarpeValetteT}
{\sc Bachir Bekka, Pierre de~la Harpe, and Alain Valette}.
\newblock {\em {Kazhdan's property ({T})}}, volume~11 of {\em New Mathematical
  Monographs}.
\newblock Cambridge University Press, Cambridge, 2008.

\bibitem{BestvinaFujiwara}
{\sc Mladen Bestvina and Koji Fujiwara}.
\newblock Bounded cohomology of subgroups of mapping class groups.
\newblock {\em Geom. Topol.} {\bf 6}, 69--89 (electronic), 2002.

\bibitem{BestvinaHandel92}
{\sc Mladen Bestvina and Michael Handel}.
\newblock Train tracks and automorphisms of free groups.
\newblock {\em Ann. of Math. (2)} {\bf 135},  No.~1, 1--51, 1992.

\bibitem{Birman69}
{\sc Joan~S. Birman}.
\newblock On braid groups.
\newblock {\em Comm. Pure Appl. Math.} {\bf 22}, 41--72, 1969.

\bibitem{bfg}
{\sc Collin Bleak, Alexander Fel{\cprime}shtyn, and Daciberg~L.
  Gon{\c{c}}alves}.
\newblock Twisted conjugacy classes in {R}. {T}hompson's group {$F$}.
\newblock {\em Pacific J. Math.} {\bf 238},  No.~1, 1--6, 2008.

\bibitem{Bowditch}
{\sc Brian~H. Bowditch}.
\newblock Intersection numbers and the hyperbolicity of the curve complex.
\newblock {\em J. Reine Angew. Math.} {\bf 598}, 105--129, 2006.

\bibitem{Brauer63}
{\sc Richard Brauer}.
\newblock Representations of finite groups.
\newblock In {\em Lectures on {M}odern {M}athematics, {V}ol. {I}}, pages
  133--175. Wiley, New York, 1963.

\bibitem{BridsonHarpe2004JFA}
{\sc Martin~R. Bridson and Pierre de~la Harpe}.
\newblock Mapping class groups and outer automorphism groups of free groups are
  {$C^*$}-simple.
\newblock {\em J. Funct. Anal.} {\bf 212},  No.~1, 195--205, 2004.

\bibitem{CohenSuciu}
{\sc Daniel~C. Cohen and Alexander~I. Suciu}.
\newblock Homology of iterated semidirect products of free groups.
\newblock {\em J. Pure Appl. Algebra} {\bf 126},  No.~1-3, 87--120, 1998.

\bibitem{Cohn1966IHES}
{\sc P.~M. Cohn}.
\newblock On the structure of the {${\rm GL}_{2}$} of a ring.
\newblock {\em Inst. Hautes \'Etudes Sci. Publ. Math.} {\bf },  No.~30, 5--53,
  1966.

\bibitem{CurtisReiner}
{\sc Charles~W. Curtis and Irving Reiner}.
\newblock {\em Representation theory of finite groups and associative
  algebras}.
\newblock Pure and Applied Mathematics, Vol. XI. Interscience Publishers, a
  division of John Wiley \& Sons, New York-London, 1962.

\bibitem{DGO}
{\sc F.~{Dahmani}, V.~{Guirardel}, and D.~{Osin}}.
\newblock {Hyperbolically embedded subgroups and rotating families in groups
  acting on hyperbolic spaces}.
\newblock E-print arxiv, 1111.7048, 2011.

\bibitem{Harpe88CRAS}
{\sc Pierre de~la Harpe}.
\newblock Groupes hyperboliques, alg\`ebres d'op\'erateurs et un th\'eor\`eme
  de {J}olissaint.
\newblock {\em C. R. Acad. Sci. Paris S\'er. I Math.} {\bf 307},  No.~14,
  771--774, 1988.

\bibitem{Harpe2007simple}
{\sc Pierre de~la Harpe}.
\newblock On simplicity of reduced {$C^\ast$}-algebras of groups.
\newblock {\em Bull. Lond. Math. Soc.} {\bf 39},  No.~1, 1--26, 2007.

\bibitem{DePe2009}
{\sc Karel Dekimpe, Bram De~Rock, and Pieter Penninckx}.
\newblock The {$R_\infty$} property for infra-nilmanifolds.
\newblock {\em Topol. Methods Nonlinear Anal.} {\bf 34},  No.~2, 353--373,
  2009.

\bibitem{DePe2011}
{\sc Karel Dekimpe and Pieter Penninckx}.
\newblock The finiteness of the {R}eidemeister number of morphisms between
  almost-crystallographic groups.
\newblock {\em J. Fixed Point Theory Appl.} {\bf 9},  No.~2, 257--283, 2011.

\bibitem{DixmierEng}
{\sc J.~Dixmier}.
\newblock {\em {C*}-Algebras}.
\newblock North-Holland, Amsterdam, 1982.

\bibitem{Dold83Inv}
{\sc A.~Dold}.
\newblock Fixed point indices of iterated maps.
\newblock {\em Invent. Math.} {\bf 74},  No.~3, 419--435, 1983.

\bibitem{Ershov2011}
{\sc Mikhail Ershov}.
\newblock Kazhdan quotients of {G}olod-{S}hafarevich groups.
\newblock {\em Proc. Lond. Math. Soc. (3)} {\bf 102},  No.~4, 599--636, 2011.

\bibitem{ErshovSurv}
{\sc Mikhail Ershov}.
\newblock {G}olod-{S}hafarevich groups: a survey.
\newblock Preprint, 2012.
\newblock http: // people.virginia.edu/ \~{}mve2x/ Research/gssurvey0114.pdf.

\bibitem{FadellNeuwirth62}
{\sc Edward Fadell and Lee Neuwirth}.
\newblock Configuration spaces.
\newblock {\em Math. Scand.} {\bf 10}, 111--118, 1962.

\bibitem{FelshB}
{\sc A.~Fel'shtyn}.
\newblock Dynamical zeta functions, {N}ielsen theory and {R}eidemeister
  torsion.
\newblock {\em Mem. Amer. Math. Soc.} {\bf 147},  No.~699, xii+146, 2000.

\bibitem{FelPOMI}
{\sc A.~Fel'shtyn}.
\newblock The {R}eidemeister number of any automorphism of a {G}romov
  hyperbolic group is infinite.
\newblock {\em Zap. Nauchn. Sem. S.-Peterburg. Otdel. Mat. Inst. Steklov.
  (POMI)} {\bf 279},  No.~6 (Geom. i Topol.), 229--240, 250, 2001.

\bibitem{FelHill}
{\sc A.~Fel'shtyn and R.~Hill}.
\newblock The {R}eidemeister zeta function with applications to {N}ielsen
  theory and a connection with {R}eidemeister torsion.
\newblock {\em $K$-Theory} {\bf 8},  No.~4, 367--393, 1994.

\bibitem{FelIndTro}
{\sc A.~Fel'shtyn, F.~Indukaev, and E.~Troitsky}.
\newblock Twisted {B}urnside theorem for two-step torsion-free nilpotent
  groups.
\newblock In {\em {C}*-algebras and elliptic theory. {II}}, Trends in Math.,
  pages 87--101. Birkh\"auser, 2008.

\bibitem{FelLeonTro}
{\sc A.~Fel'shtyn, Yu. Leonov, and E.~Troitsky}.
\newblock Twisted conjugacy classes in saturated weakly branch groups.
\newblock {\em Geometriae Dedicata} {\bf 134}, 61--73, 2008.

\bibitem{FelTro}
{\sc A.~Fel'shtyn and E.~Troitsky}.
\newblock A twisted {B}urnside theorem for countable groups and {R}eidemeister
  numbers.
\newblock In C.~Consani and M.~Marcolli, editors, {\em Noncommutative Geometry
  and Number Theory}, pages 141--154. Vieweg, Braunschweig, 2006.

\bibitem{polyc}
{\sc A.~Fel{\cprime}shtyn and E.~Troitsky}.
\newblock Twisted {B}urnside-{F}robenius theory for discrete groups.
\newblock {\em J. Reine Angew. Math.} {\bf 613}, 193--210, 2007.

\bibitem{FelTroVer}
{\sc A.~Fel'shtyn, E.~Troitsky, and A.~Vershik}.
\newblock Twisted {B}urnside theorem for type {II}${}_1$ groups: an example.
\newblock {\em Math. Res. Lett.} {\bf 13},  No.~5, 719--728, 2006.

\bibitem{f07}
{\sc Alexander Fel'shtyn}.
\newblock New directions in {N}ielsen-{R}eidemeister theory.
\newblock {\em Topology Appl.} {\bf 157},  No.~10-11, 1724--1735, 2010.

\bibitem{FelGon08Progress}
{\sc Alexander Fel{\cprime}shtyn and Daciberg~L. Gon{\c{c}}alves}.
\newblock The {R}eidemeister number of any automorphism of a
  {B}aumslag-{S}olitar group is infinite.
\newblock In {\em Geometry and dynamics of groups and spaces}, volume 265 of
  {\em Progr. Math.}, pages 399--414. Birkh\"auser, Basel, 2008.

\bibitem{dfg}
{\sc Alexander Fel'shtyn and Daciberg~L. Gon{\c{c}}alves}.
\newblock Twisted conjugacy classes in symplectic groups, mapping class groups
  and braid groups.
\newblock {\em Geom. Dedicata} {\bf 146}, 211--223, 2010.
\newblock With an appendix written jointly with Francois Dahmani.

\bibitem{FelGon2011Q}
{\sc Alexander Fel'shtyn and Daciberg~L. Gon{\c{c}}alves}.
\newblock Reidemeister spectrum for metabelian groups of the form
  {$Q^n\rtimes\Bbb Z$} and {$\Bbb Z[1/p]^n\rtimes\Bbb Z$}, {$p$} prime.
\newblock {\em Internat. J. Algebra Comput.} {\bf 21},  No.~3, 505--520, 2011.

\bibitem{feltroKT}
{\sc Alexander Fel{\cprime}shtyn and Evgenij Troitsky}.
\newblock Geometry of {R}eidemeister classes and twisted {B}urnside theorem.
\newblock {\em J. K-Theory} {\bf 2},  No.~3, 463--506, 2008.

\bibitem{FormanekProcessi1992}
{\sc Edward Formanek and Claudio Procesi}.
\newblock The automorphism group of a free group is not linear.
\newblock {\em J. Algebra} {\bf 149},  No.~2, 494--499, 1992.

\bibitem{NunoParis2006}
{\sc Nuno Franco and Luis Paris}.
\newblock On a theorem of {A}rtin. {II}.
\newblock {\em J. Group Theory} {\bf 9},  No.~6, 731--751, 2006.

\bibitem{Fujiwara98}
{\sc Koji Fujiwara}.
\newblock The second bounded cohomology of a group acting on a
  {G}romov-hyperbolic space.
\newblock {\em Proc. London Math. Soc. (3)} {\bf 76},  No.~1, 70--94, 1998.

\bibitem{FujiOhshi2002}
{\sc Koji Fujiwara and Ken'ichi Ohshika}.
\newblock The second bounded cohomology of 3-manifolds.
\newblock {\em Publ. Res. Inst. Math. Sci.} {\bf 38},  No.~2, 347--354, 2002.

\bibitem{Furstenberg}
{\sc Harry Furstenberg}.
\newblock A {P}oisson formula for semi-simple {L}ie groups.
\newblock {\em Ann. of Math. (2)} {\bf 77}, 335--386, 1963.

\bibitem{GabLeviLus98}
{\sc D.~Gaboriau, G.~Levitt, and M.~Lustig}.
\newblock A dendrological proof of the {S}cott conjecture for automorphisms of
  free groups.
\newblock {\em Proc. Edinburgh Math. Soc. (2)} {\bf 41},  No.~2, 325--332,
  1998.

\bibitem{GabJaegLevittLustig98Duke}
{\sc Damien Gaboriau, Andre Jaeger, Gilbert Levitt, and Martin
Lustig}.
\newblock An index for counting fixed points of automorphisms of free groups.
\newblock {\em Duke Math. J.} {\bf 93},  No.~3, 425--452, 1998.

\bibitem{Gantmacher}
{\sc Felix Gantmacher}.
\newblock Canonical representation of automorphisms of a complex semi-simple
  {L}ie group.
\newblock {\em Rec. Math. (Moscou)} {\bf 5(47)}, 101--146, 1939.

\bibitem{goncalves}
{\sc D.~Gon{\c{c}}alves}.
\newblock The coincidence {R}eidemeister classes of maps on nilmanifolds.
\newblock {\em Topol. Methods Nonlinear Anal.} {\bf 12},  No.~2, 375--386,
  1998.

\bibitem{go:nil1}
{\sc D.~Gon{\c{c}}alves}.
\newblock The coincidence {R}eidemeister classes on nilmanifolds and nilpotent
  fibrations.
\newblock {\em Topology and Its Appl.} {\bf 83}, 169--186, 1998.

\bibitem{gowon}
{\sc D.~Gon{\c{c}}alves and P.~Wong}.
\newblock Twisted conjugacy classes in exponential growth groups.
\newblock {\em Bull. London Math. Soc.} {\bf 35},  No.~2, 261--268, 2003.

\bibitem{GoncGuasc034}
{\sc D.~L. Gon{\c{c}}alves and J.~Guaschi}.
\newblock On the structure of surface pure braid groups.
\newblock {\em J. Pure Appl. Algebra} {\bf 186},  No.~2, 187--218, 2004.
\newblock Corrected reprint of: ``On the structure of surface pure braid
  groups'' [J. Pure Appl. Algebra {{\bf{1}}82} (2003), no. 1, 33--64;
  MR1977999].

\bibitem{GK}
{\sc Daciberg Gon{\c{c}}alves and Dessislava~Hristova
Kochloukova}.
\newblock Sigma theory and twisted conjugacy classes.
\newblock {\em Pacific J. Math.} {\bf 247},  No.~2, 335--352, 2010.

\bibitem{gowon1}
{\sc Daciberg Gon{\c{c}}alves and Peter Wong}.
\newblock Twisted conjugacy classes in wreath products.
\newblock {\em Internat. J. Algebra Comput.} {\bf 16},  No.~5, 875--886, 2006.

\bibitem{GoWon09Crelle}
{\sc Daciberg Gon{\c{c}}alves and Peter Wong}.
\newblock Twisted conjugacy classes in nilpotent groups.
\newblock {\em J. Reine Angew. Math.} {\bf 633}, 11--27, 2009.

\bibitem{GrFA}
{\sc R.I. Grigorchuk}.
\newblock On {B}urnside's problem on periodic groups.
\newblock {\em Funct. Anal. Appl.} {\bf 14}, 41--43, 1980.

\bibitem{Grossman}
{\sc Edna~K. Grossman}.
\newblock On the residual finiteness of certain mapping class groups.
\newblock {\em J. London Math. Soc. (2)} {\bf 9}, 160--164, 1974/75.

\bibitem{Groth}
{\sc A.~Grothendieck}.
\newblock Formules de {N}ielsen-{W}ecken et de {L}efschetz en g\'eom\'etrie
  alg\'ebrique.
\newblock In {\em S\'eminaire de G\'eom\'etrie Alg\'ebrique du {B}ois-{M}arie
  1965-66. {SGA} 5}, volume 569 of {\em Lecture Notes in Math.}, pages
  407--441. Springer-Verlag, Berlin, 1977.

\bibitem{GuSi}
{\sc N.~Gupta and S.~Sidki}.
\newblock On the {B}urnside problem for periodic groups.
\newblock {\em Math. Z.} {\bf 182}, 385--388, 1983.

\bibitem{Hall}
{\sc Ph. Hall}.
\newblock A characteristic property of soluble groups.
\newblock {\em J. London Math. Soc.} {\bf 12}, 198--200, 1937.

\bibitem{Hempel3mds}
{\sc John Hempel}.
\newblock {\em {$3$}-{M}anifolds}.
\newblock Princeton University Press, Princeton, N. J., 1976.
\newblock Ann. of Math. Studies, No. 86.

\bibitem{Hempel}
{\sc John Hempel}.
\newblock Residual finiteness for {$3$}-manifolds.
\newblock In {\em Combinatorial group theory and topology ({A}lta, {U}tah,
  1984)}, volume 111 of {\em Ann. of Math. Stud.}, pages 379--396. Princeton
  Univ. Press, Princeton, NJ, 1987.

\bibitem{HNN}
{\sc Graham Higman, B.~H. Neumann, and Hanna Neumann}.
\newblock Embedding theorems for groups.
\newblock {\em J. London Math. Soc.} {\bf 24}, 247--254, 1949.

\bibitem{Indukaev}
{\sc F.~K. Indukaev}.
\newblock The twisted {B}urnside theory for the discrete {H}eisenberg group and
  for the wreath products of some groups.
\newblock {\em Vestnik Moskov. Univ. Ser. I Mat. Mekh.} {\bf },  No.~6, 9--17,
  71, 2007.
\newblock translation in Moscow Univ. Math. Bull. 62 (2007), no. 6, 219--227.

\bibitem{NIvanov}
{\sc N.~V. Ivanov}.
\newblock Foundations of the theory of bounded cohomology.
\newblock {\em Zap. Nauchn. Sem. Leningrad. Otdel. Mat. Inst. Steklov. (LOMI)}
  {\bf 143}, 69--109, 177--178, 1985.
\newblock Studies in topology, V. (English translation: J. Math. Sci., 37
  (1987), No. 3, 1090--1115).

\bibitem{Jabara}
{\sc Enrico Jabara}.
\newblock Automorphisms with finite {R}eidemeister number in residually finite
  groups.
\newblock {\em J. Algebra} {\bf 320},  No.~10, 3671--3679, 2008.

\bibitem{Jiang}
{\sc B.~Jiang}.
\newblock {\em Lectures on {N}ielsen Fixed Point Theory}, volume~14 of {\em
  Contemp. Math.}
\newblock Amer. Math. Soc., Providence, RI, 1983.

\bibitem{juhasz}
{\sc A.~Juh\'asz}.
\newblock Twisted conjugacy in certain {A}rtin groups.
\newblock In {\em Ischia Group Theory 2010}, eProceedings, pages 175--195.
  World Scientific, 2011.

\bibitem{KaniuMark1992}
{\sc Eberhard Kaniuth and Annette Markfort}.
\newblock {The conjugation representation and inner amenability of discrete
  groups.}
\newblock {\em J. Reine Angew. Math.} {\bf 432}, 23--37, 1992.

\bibitem{Kirillov}
{\sc A.~A. Kirillov}.
\newblock {\em Elements of the Theory of Representations}.
\newblock Springer-Verlag, Berlin Heidelberg New York, 1976.

\bibitem{Kurosh}
{\sc A.~G. Kurosh}.
\newblock {\em The theory of groups}.
\newblock Translated from the Russian and edited by K. A. Hirsch. 2nd English
  ed. 2 volumes. Chelsea Publishing Co., New York, 1960.

\bibitem{Landau1903}
{\sc Edmund Landau}.
\newblock \"{U}ber die {K}lassenzahl der bin\"aren quadratischen {F}ormen von
  negativer {D}iscriminante.
\newblock {\em Math. Ann.} {\bf 56},  No.~4, 671--676, 1903.

\bibitem{LevittBaums}
{\sc G.~Levitt}.
\newblock On the automorphism group of generalised {B}aumslag-{S}olitar groups.
\newblock {\em Geom. Topol.} {\bf 11}, 473--515, 2007.

\bibitem{ll}
{\sc G.~Levitt and M.~Lustig}.
\newblock {Most automorphisms of a hyperbolic group have very simple dynamics.}
\newblock {\em Ann. Scient. \'Ec. Norm. Sup.} {\bf 33}, 507--517, 2000.

\bibitem{LutScze2011}
{\sc R.~{Lutowski} and A.~{Szczepa{\'n}ski}}.
\newblock Holonomy groups of flat manifolds with ${R}_\infty$ property.
\newblock arXiv:1104.5661, 2011.

\bibitem{malcev}
{\sc A.~I. Mal'cev}.
\newblock On the faithful representations of infinite groups by matrices.
\newblock {\em Mat. Sb. (NS)} {\bf 8(50)}, 405--422, 1940.
\newblock (in Russian. English translation: Amer. Math. Soc. Transl. (2),
  \textbf{45} (1965), 1--18).

\bibitem{MargulVinb2000JLieT}
{\sc G.~A. Margulis and {\`E}.~B. Vinberg}.
\newblock Some linear groups virtually having a free quotient.
\newblock {\em J. Lie Theory} {\bf 10},  No.~1, 171--180, 2000.

\bibitem{Markov45}
{\sc A.A. Markov}.
\newblock {\em Foundations of algebraic theory of braids}.
\newblock Number~16 in Trudy Instituta Steklova. {USSR A}cademy of {S}ciences,
  1945.
\newblock (in Russian, English summary).

\bibitem{MasurMinsky}
{\sc Howard~A. Masur and Yair~N. Minsky}.
\newblock Geometry of the complex of curves. {I}. {H}yperbolicity.
\newblock {\em Invent. Math.} {\bf 138},  No.~1, 103--149, 1999.

\bibitem{Meier}
{\sc David Meier}.
\newblock On polyfree groups.
\newblock {\em Illinois J. Math.} {\bf 28},  No.~3, 437--443, 1984.

\bibitem{mille-71}
{\sc C.~F. Miller~III}.
\newblock {\em On group-theoretic problems and their classification}.
\newblock Number~68 in Annals of Math. Studies. Princeton Univ. Press, 1971.

\bibitem{MubeenaSankaran1111.6181}
{\sc T.~{Mubeena} and P.~{Sankaran}}.
\newblock Twisted conjugacy classes in abelian extensions of certain linear
  groups.
\newblock arXiv:1111.6181, 2011.

\bibitem{MubeenaSankaran1201.4934}
{\sc T.~{Mubeena} and P.~{Sankaran}}.
\newblock Twisted conjugacy classes in lattices in semisimple {L}ie groups.
\newblock arXiv:1201.4934, 2012.

\bibitem{Nasybullov1201.6515}
{\sc T.~R. {Nasybullov}}.
\newblock Twisted conjugacy classes in special and general linear groups.
\newblock arXiv:1201.6515, 2012.

\bibitem{NewmanMorris}
{\sc Morris Newman}.
\newblock A bound for the number of conjugacy classes in a group.
\newblock {\em J. London Math. Soc.} {\bf 43}, 108--110, 1968.

\bibitem{nie1}
{\sc J.~Nielsen}.
\newblock {Untersuchungen zur Topologie der geschlossenen zweiseitigen
  Fl\"achen}.
\newblock {\em Acta Math.} {\bf 50}, 189--358, 1927.
\newblock (English transl. in Collected mathematical papers, Birkhauser, 1986).

\bibitem{Nielsen1924}
{\sc Jakob Nielsen}.
\newblock Die {I}somorphismengruppe der freien {G}ruppen.
\newblock {\em Math. Ann.} {\bf 91},  No.~3-4, 169--209, 1924.

\bibitem{olsh}
{\sc A.~Yu. Ol{\cprime}shanski{\u\i}}.
\newblock {\em Geometry of defining relations in groups}, volume~70 of {\em
  Mathematics and its Applications (Soviet Series)}.
\newblock Kluwer Academic Publishers Group, Dordrecht, 1991.
\newblock Translated from the 1989 Russian original by Yu.\ A. Bakhturin.

\bibitem{Onishik-Vinberg}
{\sc A.~L. Onishchik and {\`E}.~B. Vinberg}.
\newblock {\em Lie groups and algebraic groups}.
\newblock Springer Series in Soviet Mathematics. Springer-Verlag, Berlin, 1990.
\newblock Translated from the Russian and with a preface by D. A. Leites.

\bibitem{OsinGradient}
{\sc D.~Osin}.
\newblock Rank gradient and torsion groups.
\newblock {\em Bull. Lond. Math. Soc.} {\bf 43},  No.~1, 10--16, 2011.

\bibitem{Osin}
{\sc Denis Osin}.
\newblock Small cancellations over relatively hyperbolic groups and embedding
  theorems.
\newblock {\em Ann. of Math. (2)} {\bf 172},  No.~1, 1--39, 2010.

\bibitem{ParisRolfsen99}
{\sc L.~Paris and D.~Rolfsen}.
\newblock Geometric subgroups of surface braid groups.
\newblock {\em Ann. Inst. Fourier (Grenoble)} {\bf 49},  No.~2, 417--472, 1999.

\bibitem{Paris}
{\sc Luis Paris}.
\newblock Braid groups and {A}rtin groups.
\newblock In {\em Handbook of {T}eichm\"uller theory. {V}ol. {II}}, volume~13
  of {\em IRMA Lect. Math. Theor. Phys.}, pages 389--451. Eur. Math. Soc.,
  Z\"urich, 2009.

\bibitem{reid:re}
{\sc K.~Reidemeister}.
\newblock {Automorphismen von Homotopiekettenringen.}
\newblock {\em Math. Ann.} {\bf 112}, 586--593, 1936.

\bibitem{Remes}
{\sc V.~N. Remeslennikov}.
\newblock Conjugacy in polycyclic groups. ({R}ussian).
\newblock {\em Algebra i Logika} {\bf 8}, 712--725, 1969.

\bibitem{Romankov}
{\sc V.~Roman'kov}.
\newblock Twisted conjugacy classes in nilpotent groups.
\newblock {\em J. Pure Appl. Algebra} {\bf 215},  No.~4, 664--671, 2011.

\bibitem{SchPuc}
{\sc Jan-Christoph Schlage-Puchta}.
\newblock A $p$-group with positive rank gradient.
\newblock {\em J. Group Theory} {\bf 15},  No.~2, 261--270, 2012.

\bibitem{Scott70}
{\sc G.~P. Scott}.
\newblock Braid groups and the group of homeomorphisms of a surface.
\newblock {\em Proc. Cambridge Philos. Soc.} {\bf 68}, 605--617, 1970.

\bibitem{Scott1983BLMS}
{\sc Peter Scott}.
\newblock The geometries of {$3$}-manifolds.
\newblock {\em Bull. London Math. Soc.} {\bf 15},  No.~5, 401--487, 1983.

\bibitem{serrerepr}
{\sc Jean-Pierre Serre}.
\newblock {\em Linear representations of finite groups}.
\newblock Springer-Verlag, New York, 1977.
\newblock Translated from the second French edition by Leonard L. Scott,
  Graduate Texts in Mathematics, Vol. 42.

\bibitem{serrtrees}
{\sc Jean-Pierre Serre}.
\newblock {\em Trees}.
\newblock Springer Monographs in Mathematics. Springer-Verlag, Berlin, 2003.
\newblock Translated from the French original by John Stillwell, Corrected 2nd
  printing of the 1980 English translation.

\bibitem{Shokra}
{\sc Salahoddin Shokranian}.
\newblock {\em The {S}elberg-{A}rthur trace formula}, volume 1503 of {\em
  Lecture Notes in Mathematics}.
\newblock Springer-Verlag, Berlin, 1992.
\newblock Based on lectures by James Arthur.

\bibitem{Springer}
{\sc T.~A. Springer}.
\newblock Twisted conjugacy in simply connected groups.
\newblock {\em Transform. Groups} {\bf 11},  No.~3, 539--545, 2006.

\bibitem{TabWong}
{\sc Jennifer Taback and Peter Wong}.
\newblock Twisted conjugacy and quasi-isometry invariance for generalized
  solvable {B}aumslag-{S}olitar groups.
\newblock {\em J. Lond. Math. Soc. (2)} {\bf 75},  No.~3, 705--717, 2007.

\bibitem{Thurston}
{\sc William~P. Thurston}.
\newblock {\em Three-dimensional geometry and topology. {V}ol. 1}, volume~35 of
  {\em Princeton Mathematical Series}.
\newblock Princeton University Press, Princeton, NJ, 1997.
\newblock Edited by Silvio Levy.

\bibitem{ncrmkwb}
{\sc E.~Troitsky}.
\newblock Noncommutative {R}iesz theorem and weak {B}urnside type theorem on
  twisted conjugacy.
\newblock {\em Funct. Anal. Pril.} {\bf 40},  No.~2, 44--54, 2006.
\newblock In Russian, English translation: \emph{Funct. Anal. Appl.}
  \textbf{40} (2006), No. 2, 117--125.

\bibitem{Vinberg71IAN}
{\sc {\`E}.~B. Vinberg}.
\newblock Discrete linear groups that are generated by reflections.
\newblock {\em Izv. Akad. Nauk SSSR Ser. Mat.} {\bf 35}, 1072--1112, 1971.

\bibitem{Vogtmann2002}
{\sc Karen Vogtmann}.
\newblock Automorphisms of free groups and outer space.
\newblock In {\em Proceedings of the {C}onference on {G}eometric and
  {C}ombinatorial {G}roup {T}heory, {P}art {I} ({H}aifa, 2000)}, volume~94,
  pages 1--31, 2002.

\bibitem{Weck}
{\sc Franz Wecken}.
\newblock Fixpunktklassen. {I}.
\newblock {\em Math. Ann.} {\bf 117}, 659--671, 1941.

\bibitem{WehrfritzLinearGroups}
{\sc B.~A.~F. Wehrfritz}.
\newblock {\em Infinite linear groups. {A}n account of the group-theoretic
  properties of infinite groups of matrices}.
\newblock Springer-Verlag, New York, 1973.
\newblock Ergebnisse der Matematik und ihrer Grenzgebiete, Band 76.

\bibitem{Wehrfritz73JLMS}
{\sc B.~A.~F. Wehrfritz}.
\newblock Two examples of soluble groups that are not conjugacy separable.
\newblock {\em J. London Math. Soc. (2)} {\bf 7}, 312--316, 1973.

\bibitem{Zarelua2008Steklo}
{\sc A.~V. Zarelua}.
\newblock On congruences for the traces of powers of some matrices.
\newblock {\em Tr. Mat. Inst. Steklova} {\bf 263},  No.~Geometriya, Topologiya
  i Matematicheskaya Fizika. I, 85--105, 2008.
\newblock English translation: Proceedings of the Steklov Institute of
  Mathematics, 2008, Vol. 263, pp. 78--98.

\end{thebibliography}
\def\cprime{$'$} \def\cprime{$'$} \def\cprime{$'$} \def\cprime{$'$}
  \def\cprime{$'$} \def\cprime{$'$} \def\cprime{$'$} \def\dbar{\leavevmode\hbox
  to 0pt{\hskip.2ex \accent"16\hss}d} \def\cprime{$'$} \def\cprime{$'$}
  \def\polhk#1{\setbox0=\hbox{#1}{\ooalign{\hidewidth
  \lower1.5ex\hbox{`}\hidewidth\crcr\unhbox0}}} \def\cprime{$'$}
  \def\cprime{$'$} \def\cprime{$'$} \def\cprime{$'$}

\end{document}